\UseRawInputEncoding
\documentclass[12pt]{amsart}
\usepackage{cases}
\usepackage{amscd}
\usepackage{amsmath}
\usepackage{amsthm}
\usepackage{amssymb}
\usepackage{mathrsfs}
\usepackage{amsfonts}
\usepackage{color}

\usepackage{pifont}

\usepackage{upgreek}
\usepackage{bm}

\usepackage{indentfirst, latexsym, bm, amsmath, eufrak, amsthm}
\usepackage{amsmath}
\usepackage{amsthm}
\usepackage{amssymb}
\usepackage{mathrsfs}
\usepackage{amsfonts}

\usepackage{pifont}

\usepackage{upgreek}
\usepackage{bm}
\newcommand*{\tr}{\mathrm{tr}}
\numberwithin{equation}{section}
\newtheorem{theo}{Theorem}

\newtheorem{lem}{Lemma}
\newtheorem{mcor}{Corollary}
\newtheorem{remark}{Remark}

\newtheorem{definition}{Definition}

\begin{document}
\title[An equivalence theorem of a class of Minkowski norms...]{An equivalence theorem of a class of Minkowski norms and its applications}

\author[Huitao Feng]{Huitao Feng$^1$}
\author[Yuhua Han]{Yuhua Han}
\author[Ming Li]{Ming Li$^2$}

\address{Huitao Feng:  Chern Institute of Mathematics \& LPMC, Nankai University, Tianjin
300071, P. R. China}
\email{fht@nankai.edu.cn}
\address{Yuhua Han:  Chern Institute of Mathematics \& LPMC, Nankai University, Tianjin
300071, P. R. China}
\email{827053901@qq.com}
\address{Ming Li: Mathematical Science Research Center ,
Chongqing University of Technology,
Chongqing 400054, P. R. China }
\email{mingli@cqut.edu.cn}

\thanks{$^1$~Partially supported by NSFC (Grant No. 11221091, 11271062, 11571184, 11931007) and the Fundamental Research Funds for the General Universities and Nankai Zhide Foundation.}

\thanks{$^2$~Partially supported by NSFC (Grant No. 11501067, 11571184, 11871126) and CSC Visiting Scholar Program}

\maketitle

\begin{center}
  \textit{In memory of Professor Zhengguo Bai}
\end{center}

\begin{abstract}
In this paper, the Cartan tensors of the $(\alpha,\beta)$-norms are investigated in details. Then an equivalence theorem of $(\alpha,\beta)$-norms is proved.  As a consequence in Finsler geometry, general $(\alpha,\beta)$-metrics on smooth manifolds of dimension $n\geq4$ with vanishing Landsberg curvatures must be Berwald manifolds.
\end{abstract}


\section*{Introduction}

In Finsler geometry, there are two classes of important manifolds, namely Berwald manifolds and Landsberg manifolds. By definition, one has the following relation between them
\begin{center}
\{Berwald manifolds\} $\subseteq$ \{Landsberg manifolds\}.
\end{center}
However, people couldn't find Landsberg manifolds which are not Berwald manifolds, since L. Berwald introduced these concepts during 1920s.

A natural and long-standing problem in Finsler geometry is the following:
\begin{center}
\emph {Is there any Landsberg metric which is not of Berwald type?}.
\end{center}

In 1996, Matsumoto listed in the short survey \cite{Matsumoto3} many results which shows Landsberg spaces under certain conditions, for example vanishing Douglas curvature,  having C-reducible metrics..., are Berwald. Finally, he listed it as an open problem in his paper. 


In 2005, D. Bao suggested to name Landsberg metrics that are not Berwald type as \emph{unicorns} in \cite{Bao}. Since then this problem is referred to as the \emph{unicorn problem}.
Many authors make contributions on this problem(c.f. \cite{Aikou10,BaIlKi,BaMa,Bao,Cra1,Cra2,Li1,LiZ,Matveev09,Muz,Szabo08,Szabo09,Tor,ZL,ZohM,ZouCheng14}).

The unicorn problem has no global requirements for the base manifold, thus it looks like a local problem.  However the progress on the unicorn problem suggests that it is a fiberwise global problem.

In \cite{Li1}, one of the author finds a way of using equivalence theorems of Minkowski spaces to investigate the unicorn problem. One of the results in  \cite{Li1} is that a Landsberg manifold of dimension $n\geq3$ with vanishing mean Berwald curvature $\mathbf{E}$ must be Berwald. Furthermore, Crampin \cite{Cra2} improved this result by replacing the condition $\mathbf{E}=0$ by the Berwald scalar curvature $\mathbf{e}:=\tr\mathbf{E}=0$. Another way to obtain Crampin's result can also be found in \cite{LiZ}. The crucial fact behind these results is that the unit tangent spheres of a Finsler manifold are convex and compact.

In literatures, an $(\alpha,\beta)$-metric of $M$ is defined by a Riemannian metric $\boldsymbol{\alpha}$, a 1-form $\boldsymbol{\beta}$ and a $C^{\infty}$ function $\phi(x)$ of a single real variable in the following form
\begin{equation*}
  \mathbf{F}(x,y)=\boldsymbol{\alpha}(x,y)\phi\left(\frac{\boldsymbol{\beta}(x,y)}{\boldsymbol{\alpha}(x,y)}\right), \quad \forall (x,y)\in TM.
\end{equation*}
Restricted on each tangent space of $M$, $\mathbf{F}$ induces $(\alpha,\beta)$-norms (see Section 2 for details) with the same function $\phi$.

If the Finsler metrics are allowed to define not on the entire tangent bundle, in 2006 Asanov \cite{Asanov06a,Asanov06b} constructed unicorns by using  $(\alpha,\beta)$-metrics. Asanov's unicorns are defined on the tangent bundle except from a real line subbundle of $TM$. However, in 2009 Shen \cite{Shen09} calculated the spray coefficients and the Landsberg curvature of an $(\alpha,\beta)$-metric with the assistance of the computer program Maple. Then he was able to prove that a Landsberg manifold of dimension $n\geq3$ with $(\alpha,\beta)$-metrics (on the entire tangent bundles) is Berwald by solving several nontrivial equations.

In this paper, we consider more general metrics given by 
\begin{equation}\label{gab}
  \mathbf{F}(x,y)=\boldsymbol{\alpha}(x,y)\phi\left(x,\frac{\boldsymbol{\beta}(x,y)}{\boldsymbol{\alpha}(x,y)}\right), \quad \forall (x,y)\in TM,
\end{equation}
where $\phi(x,s)$ is a $C^{\infty}$ function of $s$ defined on an interval $I_x$, for any fixed $x\in M$, and $\phi(x,s)$ is smooth on $\{(x,s)\in M\times \mathbb{R}|s\in I_x\}.$
Following \cite{YuZh}, a Finsler metric (\ref{gab}) is called a \emph{general $(\alpha,\beta)$-metric}.  For this type of metrics, we have the following result.
\begin{theo}\label{theo 1}
Let $M$ be a smooth manifold with dimension $n\geq4$. Let $\mathbf{F}$ be a general $(\alpha,\beta)$-metric defined on $TM$. If $M$ is a Landsberg manifold then it is Berwald.
\end{theo}

Recently, following Shen's strategy in \cite{Shen09}, it is proved in \cite{ZWL} that Landsberg manifolds of dimension $n\geq3$ must be Berwald if their metrics belong to a special class of general $(\alpha,\beta)$-metrics of the following type
\begin{equation*}
  \mathbf{F}(x,y)=\boldsymbol{\alpha}(x,y)\phi\left(\|\boldsymbol{\beta}\|^2_{\boldsymbol{\alpha}},\frac{\boldsymbol{\beta}(x,y)}{\boldsymbol{\alpha}(x,y)}\right), \quad \forall (x,y)\in TM,
\end{equation*}
where $\phi$ is a smooth function of two real variables. This result is a special case of Theorem \ref{theo 1} when the dimension $n\geq4$.

For general $(\alpha,\beta)$-metrics it seems too complicated and difficult to discuss the unicorn problem by Shen's approach. The strategy in our paper is establishing an equivalence theorem for Minkowski spaces with $(\alpha,\beta)$-norms and then using nonlinear parallel transports.

The paper contains four sections and one appendix and is organized as follows. In Sect. 1, we give a very brief introduction to Finsler geometry and a description of our strategy to deal with the unicorn problem. In Sect. 2, the Cartan tensor of an $(\alpha,\beta)$-norm is formulated using the notations introduced by Shen \cite{Shen09}. In Sect. 3, by using the centroaffine differential geometric structure of indicatrices of $(\alpha,\beta)$-norms, we obtain an equivalence theorem. In Sect. 4, by applying the equivalence theorem to the parallel transports of Finsler manifolds, we give a proof of Theorem \ref{theo 1}. In Appendix,  we investigate an ODE which is critical for our discussion. 

\medskip

{\bf Acknowledgements.}  The third author
would like to express his appreciation to Professor Guofang
Wang for his warm hospitality and discussions on Mathematics, while he stayed in
Mathematical Institute Albert Ludwigs University Freiburg in 2018.

\section{Preliminaries of Finsler geometry}

\subsection{Centroaffine normalization of a nondegenerate hypersurface}
\

In this paper, we emphasis  the centroaffine differential geometry of the indicatrix of a Minkowski norm. For this reason, we would like to review some necessary
facts of an affine hypersurface with the centroaffine normalization. One refers to \cite{LSZ,NS,SSV} for details.

Let $V$ be a real vector space of dimensional $n$ with a chosen orientation.
Let $V^*$ be the dual space of $V$, and $\langle~,~\rangle:V\times V^*\rightarrow \mathbb{R}$ the canonical pairing.
$V$ has a smooth manifold structure and a flat affine connection $\bar{\nabla}$ on $TV$.


Let $x:M\rightarrow V$ be an immersed connected oriented smooth manifold $M$ of dimension $n-1$.
Then for each point of $p\in M$, $dx(T_pM)$ is an $n-1$ dimensional subspace of $T_{x(p)}V$,
and defines an one dimensional subspace $C_pM=\{v^*_p\in V^*|\ker v^*_p=dx(T_pM)\}\subset T^*_{x(p)}V$.
The trivial line bundle $CM=\bigcup_pC_pM$ is called the conormal line bundle of $x$. 

Let $Y$ be a nowhere vanishing section of $CM$.
If ${\rm rank}(dY,Y)=n$, then $x$ is called a nondegenerate hypesurface. The nondegenerate property is independent to the choice of the conormal field $Y$.
We will only concern nondegenerate hypersurfaces.
Let $y:M\rightarrow V$ be the vector filed defined by $y(p)=-x(p)$, $\forall p\in M$.
If $\langle Y, y\rangle=1$, then the pair $\{Y,y\}$ is called the centroaffine normalization of $x$.
We state the structure equations of $x(M)$ with respect to the centroaffine normalization $\{Y,y\}$ as following,
\begin{align*}
\bar{\nabla}_vy=dy(v)=-dx(v),
\end{align*}
\begin{align*}
\bar{\nabla}_vdx(w)=dx(\nabla_vw)+\mathsf{h}(v,w)y,
\end{align*}
\begin{align*}
\bar{\nabla}^*_vdY(w)=dY(\nabla^*_vw)-\mathsf{h}(v,w)Y,
\end{align*}
where $\mathsf{h}$ is a nondegenerate symmetric $(0,2)$-tensor and called the induced affine metric,
$\nabla$ and $\nabla^*$ are torsion free affine connections.
These geometric quantities satisfy
\begin{align}
d\mathsf{h}(v_1,v_2)=\mathsf{h}(\nabla v_1,v_2)+\mathsf{h}(v_1,\nabla^*v_2). \label{conjugate connections}
\end{align}

Then the triple $\{\nabla,\mathsf{h},\nabla^*\}$  are called conjugate connections.
For any triple of conjugate connections $\{\nabla,\mathsf{h},\nabla^*\}$, one can define
$C=\frac{1}{2}(\nabla-\nabla^*)\in\Omega^1(M,{\rm End}(TM))$.
By (\ref{conjugate connections}), one can prove that $C=-\frac{1}{2}\mathsf{h}^{-1}\nabla \mathsf{h}$. Moreover, the $(0,3)$-tensor $\hat{C}:=\mathsf{h}\circ C$ is totally symmetric and called the cubic form of $\{\nabla,\mathsf{h},\nabla^*\}$. The Tchebychev form $\hat{T}$ is defined as the normalized trace of $C$,
\begin{align*}
\hat{T}=\frac{1}{n-1}{\rm tr}C.
\end{align*}
The Tchebychev field $T$ is the dual of $\hat{T}$ with respect to $\mathsf{h}$.



We have the following equations from the integrability of an affine hypersurface $x(M)$ with the centroaffine normalization $\{Y,y\}$,
\begin{align}
R(U,V)W=-\left[\mathsf{h}(U,W)V-\mathsf{h}(V,W)U\right]+\left[C(C(U)W)V-C(C(V)W)U\right],\label{Rh intergral condition c}
\end{align}
and
\begin{align*}
(\nabla^{\mathsf{h}}_UC)(V)W-(\nabla^{\mathsf{h}}_VC)(U)W=0,
\end{align*}
where $\nabla^{\mathsf{h}}$ and $R$ are the Levi-Civita connection and curvature tensor of $\mathsf{h}$, respectively.

For the fundamental theorem of centroaffine differential geometry of hypersurfaces, one refers to \cite{SSV}.

\subsection{A Minkowski norm and its indicatrix with centroaffine normalization}
\

Let $V$ be a $\mathbf{R}$-vector space of dimension $n$. 
Let $\mathbf{F}:V\to [0,+\infty)$ be a function which satisfies:
\begin{enumerate}
\item[(i)] $\mathbf{F}$ is continuous on $V$, and smooth on $V_0:=V\setminus\{0\}$;
\item[(ii)] $\mathbf{F}(\lambda y)=\lambda \mathbf{F}(y)$, \quad $\forall y\in V,\quad \forall \lambda\in \mathbb{R}^+$;
\item[(iii)] $\mathbf{F}$ is strongly convex, i.e., $(g_{ij}):=\left(\frac{1}{2}[F]^2_{y^iy^j}\right)$ is positive definite on $V_0$.
\end{enumerate}

By the strongly convexity of $\mathbf{F}$, $\bar{g}=g_{ij}dy^i\otimes dy^j$ induces a Riemannian metric on $V_0$.  The Cartan tensor is defined by
$$A:=A_{ijk}dy^i\otimes dy^j\otimes dy^k:=\frac{F}{2}\frac{\partial g_{ij}}{\partial y^k}dy^i\otimes dy^j\otimes dy^k.$$ The Cartan form is the trace of $A$ with respect to $\bar{g}$
$$\eta:=g^{jk}A_{ijk}dy^i=:A_idy^i,$$
where $(g^{ij})=(g_{ij})^{-1}$. One can easily get another formula of the Cartan form
\begin{equation}\label{eta}
  \eta=\frac{\mathbf{F}}{2}d\log\det(g_{ij}).
\end{equation}
For the following study in our paper, we also introduce here the angular metric $\mathbf{h}=h_{ij}dy^i\otimes dy^j:=\bar{g}-d\mathbf{F}\otimes d\mathbf{F}$.

In \cite{Li1}, we established the correspondence between the Minkowski spaces and the convex hypersurfaces with the origin in its interior.

 The indicatrix $\mathbf{I}_{\mathbf{F}}:=\{y\in V|\mathbf{F}(y)=1\}$ is a non-degenerate hypersurface in the context of affine differential geometry.
$\{-d\mathbf{F},-y\}$ is the centroaffine normalization of $i:\mathbf{I}_{\mathbf{F}}\hookrightarrow V$.

  Let $\mathsf{h}$, $\hat{C}$ and $\hat{T}$ be the induced affine metric, the cubic form and the Tchebychev form of $i:\mathbf{I}_{\mathbf{F}}\hookrightarrow V$ with respect to the the centroaffine normalization $\{-d\mathbf{F},-y\}$. One can prove that
  \begin{equation}
    \mathsf{h}=i^*\bar{g}=i^*\mathbf{h},\qquad \hat{C}=-i^*A,\qquad \hat{T}=-\frac{1}{n-1}i^*\eta.
  \end{equation}

Conversely, let $i:M\hookrightarrow V$ be an embedded closed non-degenerate hypersurface. If the induced affine metric $\mathsf{h}$ is positive definite with respect to the centroaffine normalization on $M$, then there exists an unique Minkowski norm $\mathbf{F}:V\to [0,+\infty)$ satisfies \textnormal{(i)-(iii) }and $i(M)=\mathbf{I}_{\mathbf{F}}$.
The function $\mathbf{F}$ may be referred as the potential function of the hypersurface $i:M\hookrightarrow V$ with respect to the centroaffine normalization.

\begin{remark}
  \textnormal{The fundamental tensor $\bar{g}$ gives rise to a Hessian metric on $V_0$. The affine differential geometry of the level sets of the potential function of a Hessian domain in $\mathbf{R}^{n}$ is investigated in \cite{Shima}.}
\end{remark}

By using Schneider's rigidity theorem \cite{Schneider1}, the following equivalence theorem is proved in \cite{Li1}.
\begin{theo}[\cite{Li1}]\label{theo equivalence}
Let $(V_1,\mathbf{F}_1)$ and $(V_2,\mathbf{F}_2)$ be two Minkowski spaces of dimension $n\geq3$, respectively.
Let $f:V_1\rightarrow V_2$ be a norm preserving map which is a diffeomorphism on $(V_1)_0:=V_1\setminus\{0\}$, and satisfies
\begin{align*}
f(tv)=tf(v), \quad \forall v\in V_1,~\forall t>0.
\end{align*}
Then $\bar{g}_1=f^*\bar{g}_2$ and $\eta_1=f^*\eta_2$ if and only if  $f$ is linear.
\end{theo}

In the case of dimension $n=2$, the equivalence theorem has been proved in \cite{BaoChernShen}.

\subsection{The Landsberg curvature and the Berwald curvature}
\

Let $M$ be a smooth manifold of dimension $n$. Set $TM_0=TM\setminus\{0\}$, where $0$ denotes the zero section of $TM$.  A Finsler structure $\mathbf{F}$ is a function on the tangent bundle $TM$ which is smooth on $TM_0$, such that $\mathbf{F}|_{T_xM}$ is a Minkowski norm restricted on each fiber $T_xM$, $\forall x\in M$. $\mathbf{F}$ is defined by a Riemannian metric if and only if the Minkowski norms on the fibers are induced by inner products.

For a Finsler manifold $(M,\mathbf{F})$, the tangent bundle of $TM_0$ admits a nature splitting
\begin{equation}\label{splitting}
  T(TM_0)=H(TM_0)\oplus V(TM_0),
\end{equation}
where $V(TM_0)$ is the distribution defined by the tangent spaces of the fibers of $TM_0$, and $H(TM_0)_{(x,y)}$ is isomorphic to $T_xM$ for any $(x,y)\in TM_0$. The splitting of $T(TM_0)$ defines a natural almost complex structure $J:T(TM_0)\to T(TM_0)$.

The fundamental tensor $g=g_{ij}dx^i\otimes dx^j$ defines a Euclidean metric on the horizontal subbundle $H(TM_0)$. The Chern connection $\nabla^{\rm Ch}$ is the unique linear connection on the Euclidean bundle $(H(TM_0),g)$, which is torsion free and almost metric compatible. In \cite{FL}, we proved that the Chern connection is just the Bott connection on $H(TM_0)$ in the theory of foliation (c.f. \cite{Zhang}), and the symmetrization of the Chern connection $ \widehat{\nabla }^{\rm Ch}$ is the Cartan connection. Also in \cite{FL} we introduced the data
\begin{align*}
H:=\widehat{\nabla}^{\rm Ch}-\nabla^{\rm Ch},\quad \eta={\rm tr}[H]~\in\Omega^1(TM_0),
\end{align*}
which are named the Cartan endomorphism and the Cartan-type form, respectively.
It is easy to show that the restriction of $\eta$ on the fibers of $TM_0$ are the Cartan forms of the corresponding Minkowski spaces.

\begin{remark}
In literature, Cartan from $I$ is defined as $I=\sum_{i,j,k}g^{jk}A_{ijk}dx^i$. It is easy to check $I=J^*\eta$, where $J^*$ denote the dual of the almost complex structure $J$ mentioned above. This is the reason why we call $\eta$ the Cartan-type form.
The differential forms $\eta$ and $I$ behavior differently. A simple computation shows that $dI=0$ implies that $M$ is a Riemannian manifold.
However there are non-Riemannian metrics such that $d\eta=0$ (c.f. \cite{SL}). The Cartan-type form $\eta$ is crucial for our study.
\end{remark}

Let $R^{\rm Ch}=\left(\nabla^{\rm Ch}\right)^2$  be the curvature of $\nabla^{\rm Ch}$.
From the torsion freeness, the curvature forms has no pure vertical differential form
\begin{align*}
R^{\rm Ch}=R+P,
\end{align*}
where the $(2,0)$ part $R$ is called the Chern-Riemann curvature and the $(1,1)$ part is called the Chern curvature.
The Landsberg curvature is defined as
$$\mathbf{L}:=-P(\mathbf{e}_n,\cdot).$$

If a Finsler manifold satisfies $P=0$ (or $\mathbf{L}=0$), then it is called a Berwald manifold (or Landsberg manifold).

If $\mathbf{L}=0$, using the explicit formulae of the connections and the curvature tensors under natural coordinate systems (c.f. \cite{BaoChernShen}. pp. 27-67), one can prove that the Berwald curvature satisfies
\begin{equation*}
  \mathbf{B}=P/\mathbf{F}.
\end{equation*}
In this case, the mean Berwald curvature or the $\mathbf{E}$-curvature equals to $\mathbf{F}^{-1}\tr P$.

\subsection{Basic properties of nonlinear parallel transport}
\

In this subsection, we review some facts of the nonlinear parallel transport. One refers to \cite{Aikou10,Bao,ChernShen} for details.

Let $(M,\mathbf{F})$ be a Finsler manifold of dimension $n$. Let $\sigma:[0,1]\rightarrow M$ be a smooth curve emanates form $\sigma(0)=p\in M$.
Without loss of generality, we can assume that $\sigma$ is contained in a local coordinate chart $(U;x^i)$ of $M$.
Let $(\sigma^1(t),\ldots,\sigma^n(t))$ be the local coordinates of $\sigma(t)$, $t\in[0,1]$.

A vector field along $\sigma$ is in fact a curve $\hat{\sigma}(t)=(\sigma(t),y(t))$ with $y(t)\in T_{\sigma(t)}M$, $t\in[0,1]$.
$\hat{\sigma}$ is called a nonlinear parallel vector filed, if it is a solution of the following ODE
\begin{align} \label{equation of parellel vector field}
\frac{dy^i}{dt}+\frac{d\sigma^j}{dt}\frac{\partial G^i}{\partial y^j}(\sigma(t),y(t))=0,\quad i=1,\ldots,n,
\end{align}
with the initial value $(\sigma(0),y(0))=(p,y)$, where the coefficients $G^i$ are defined by
\begin{align}\label{G}
G^{i}=\frac{1}{4}g^{ij}\left(\left[F^2\right]_{y^{j}x^{k}}y^{k}-\left[F^2\right]_{x^{j}}\right),\quad i=1,\ldots, n,
\end{align}
which are often called the spray coefficients. In this case, the curve $\hat{\sigma}(t)$ is called the horizontal lift of $\sigma(t)$ started from $y(0)=y$.

\begin{definition}
For any $t_0\in[0,1]$, the mapping $P_{\sigma,t_0}:T_pM\setminus\{0\}\rightarrow T_{\sigma(t_0)}M\setminus\{0\}$ is defined by
\begin{align*}
P_{\sigma,t_0}(y):=y(t_0), \quad \forall ~y\in T_pM\setminus\{0\},
\end{align*}
where $y(t)$ is the nonlinear parallel vector field along $\sigma$ with $y(0)=y$. $P_{\sigma,t_0}$ is called the nonlinear parallel transport  along $\sigma$ from $0$ to $t_0$.
\end{definition}

\begin{lem}[\cite{Aikou10,Bao,ChernShen}]\label{lemmma basic pro of npt}
Let $\sigma:[0,1]\rightarrow M$ be a smooth curve
such that the nonlinear parallel translation $P_{\sigma,t}$ is
defined for $\forall ~t\in [0,1]$. Let $p=\sigma(0)$, then
the map $P_{\sigma,t}:T_pM\setminus\{0\}\rightarrow
T_{\sigma(t)}M\setminus\{0\}$ is a norm preserving diffeomorphism,
and satisfies
\begin{align}\label{radial linear of P}
P_{\sigma,t}(\lambda y)=\lambda P_{\sigma,t}(y), \quad \forall ~\lambda>0, \forall~y\in T_pM\setminus\{0\}.
\end{align}
Therefore we can continually extend $P_{\sigma,t_0}$ to $T_pM$ by setting $P_{\sigma,t_0}(0)=0$.
\end{lem}

Let $T$ be a vertical covariant tensor field on $TM_0$, i.e. $T\equiv 0~({\rm mod}~\delta y^{1},\ldots,\delta y^{n})$, where $\{\delta y^{1},\ldots,\delta y^{n}\}$ is the local basis of $V^*(TM_0)$ with respect to the splitting (\ref{splitting}).
$T$ is called to be preserved by parallel transport along the curve $\sigma$ if
\begin{equation*}
  P_{\sigma,t}^*T_{\sigma(t)}=T_{p}\quad {\rm for}~\forall~t\in[0,\epsilon],
\end{equation*}
where $T_x=i_x^*T$ denotes the restriction of $T$ on $T_xM\setminus\{0\}$, $i_x:T_xM\hookrightarrow TM$ is the embedding mapping for any $x\in M$.

A vertical tensor is preserved by the parallel transports has the following  infinitesimal description.
\begin{lem}[\cite{Li1}]\label{T}
Let $T$ be a vertical covariant tensor field on $TM_0$.
Then $T$ is preserved by the parallel transports  if and only if
\begin{equation*}
  \mathcal{L}_{X}T\equiv 0~({\rm mod}~dx^1,\ldots,dx^n),
\end{equation*}
for any horizontal vector field $X\in\Gamma(H(TM_0))$.
\end{lem}

We modify the Cartan endomorphism $H$ to a vertical tensor field on $TM_0$ as follows
\begin{align*}
A=A_{ijk}\delta y^i\otimes \delta y^j\otimes \delta y^k.
\end{align*}
For any $p\in M$, the restriction of $\bar{g}$ and $A$ on each fiber $T_pM\setminus\{0\}$ gives the fundamental form and Cartan tensor of the Minkowski space $(T_pM, \mathbf{F}_{T_pM})$, respectively.

By Lemma \ref{T} and direct calculations, we obtain the geometric elaboration of some non-Riemannian curvatures using parallel transports.
\begin{lem}[\cite{Ichijyo78,Li1}]\label{lemma curvature parallel}
Let $(M,\mathbf{F})$ be a Finsler manifold. The following assertions are true:
\begin{enumerate}
  \item[(a)] $\mathbf{L}=0$ if and only if $\bar{g}$ is preserved by parallel transports.
  \item[(b)] $P=0$ if and only if $A$ is preserved by parallel transports.
  In this case, parallel transports are linear which is equivalent to say the spray coefficient $G^i$ is quadratic in $y$, for $i=1,\ldots,n$.
  \item[(c)] $\tr P=0$ if and only if $\eta$ is preserved by parallel transports.
\end{enumerate}
\end{lem}

By Lemma \ref{lemmma basic pro of npt} and Lemma \ref{lemma curvature parallel}, we can ascribe the unicorn problem to an equivalence problem of Minkowski spaces and finally to a rigidity problem in affine geometry.
Based on this observation and by using Theorem \ref{theo equivalence}, the following result is proved.
\begin{theo}[\cite{Li1}]\label{OK}
  Let $(M,\mathbf{F})$ be a Landsberg manifold of dimension $n\geq3$. If the mean Berwald curvature $\mathbf{E}=0$, then $P=0$.
\end{theo}
It is proved in \cite{LiZ}, $\mathbf{e}:=\tr \mathbf{E}=0$ implies $\mathbf{E}=0$ itself. Hence
\begin{theo}[\cite{Cra2,LiZ}]\label{better}
  Let $(M,\mathbf{F})$ be a Landsberg manifold of dimension $n\geq3$. If the Berwald scalar curvature $\mathbf{e}=0$, then $P=0$.
\end{theo}
Inspired by \cite{Li1}, Crampin obtained a global result for Minkowski norms, which is equivalent to a special case of Satz 3.1 in \cite{Schneider1}. Then Crampin was able to prove Theorem \ref{better} by using covariant differentiation in Finsler geometry in \cite{Cra2}. Another proof of Theorem \ref{theo equivalence} was also given in \cite{Cra2}.



\section{Some properties of $(\alpha,\beta)$-norms}

\subsection{$(\alpha,\beta)$-norms}
\

Let $\boldsymbol{\alpha}$ be a Euclidean norm, $\boldsymbol{\beta}$ a linear functional on $V$. Let $b:=\|\boldsymbol{\beta}\|_{\boldsymbol{\alpha}}$ be the length of $\boldsymbol{\beta}$ with respect to $\boldsymbol{\alpha}$. It is clear that $\frac{\boldsymbol{\beta}(y)}{\boldsymbol{\alpha}(y)}$
is positively homogenous of degree 0 and has range $[-b,b]$.  Let $\phi=\phi(s)$ be a $C^{\infty}$ positive function on an open interval $I$, such that $[-b,b]\subseteq I$. An $(\alpha,\beta)$-functional is defined as
\begin{align}\label{alpha beta norm}
\mathbf{F}=\boldsymbol{\alpha}\cdot\phi\left(\frac{\boldsymbol{\beta}}{\boldsymbol{\alpha}}\right).
\end{align}
Fixing an orthonormal basis $\{e_1,\ldots,e_n\}$ of $\boldsymbol{\alpha}$, for any $y=y^ie_i\in V_0$, we have
$$\boldsymbol{\alpha}(y)=\alpha(y^1,\ldots,y^n)=\sqrt{\delta_{ij}y^iy^j}=|y|,\quad \quad \boldsymbol{\beta}(y)=b_iy^i,$$
and
$$\frac{\boldsymbol{\beta}(y)}{\boldsymbol{\alpha}(y)}=\frac{b_iy^i}{|y|}=:s(y^1,\ldots,y^n).$$
Therefore the functional (\ref{alpha beta norm}) is expressed as
\begin{align}\label{alpha beta norm'}
\mathbf{F}(y)=\alpha\cdot\phi(s)=|y|\phi\left(\frac{b_iy^i}{|y|}\right).
\end{align}

By definition, the functional (\ref{alpha beta norm}) satisfies (i) and (ii) on $V$. The positive definite condition of (\ref{alpha beta norm}) is formulated as follows.
\begin{lem}[\cite{Shen09,ChernShen,YuZh}]\label{Lemma 1}
When $n=2$, the functional (\ref{alpha beta norm}) defines a Minkowski norm on $V$ if and only if
\begin{align}\label{conditions phi form norm n=2}
      (\phi-s\phi')+(b^2-s^2)\phi''>0,
\end{align}
holds for any $s\in[-b,b]$.

When $n\geq3$, the functional (\ref{alpha beta norm}) defines a Minkowski norm on $V$ if and only if
\begin{align}\label{conditions phi form norm}
  \phi-s\phi'>0,\quad  (\phi-s\phi')+(b^2-s^2)\phi''>0,
\end{align}
hold for any $s\in[-b,b]$.
Moreover, the determinate of $(g_{ij})$ is given by
\begin{align}\label{det}
\det{(g_{ij})}=\phi^{n+1}(\phi-s\phi')^{n-2}\left[(\phi-s\phi')+(b^2-s^2)\phi''\right].
\end{align}
\end{lem}

The following corollary is easy to verify.

\begin{mcor}\label{cor F}
Assume that $\mathbf{F}=\boldsymbol{\alpha}\left[\phi\circ\frac{{\boldsymbol{\beta}}}{\boldsymbol{\alpha}}\right]$  is an $(\alpha,\beta)$-norm with $b=\|\boldsymbol{\beta}\|_{\boldsymbol{\alpha}}$. Let $\tilde{\phi}(t)=\phi(bt)$ and $\tilde{\boldsymbol{\beta}}=b^{-1}\boldsymbol{\beta}$, then
$\mathbf{F}=\boldsymbol{\alpha}\left[\tilde{\phi}\circ\frac{\tilde{\boldsymbol{\beta}}}{\boldsymbol{\alpha}}\right]$
with $\|\tilde{\boldsymbol{\beta}}\|_{\boldsymbol{\alpha}}=1$.
\end{mcor}

\subsection{The Cartan tensor of $(\alpha,\beta)$-norms}
\

The Cartan tensor of an $(\alpha,\beta)$-norm is calculated in the following lemma. Parts of this result can be found in such as \cite{AntIM,Asanov06a,Matsumoto72,MatShi}.
\begin{lem}\label{Lemma 4}
Let $V$ be a real vector space of dimension $n\geq 3$. Let $\mathbf{F}$ be an $(\alpha,\beta)$-norm defined on $V$. The Cartan tensor of  $\mathbf{F}$ can be formulated as following
\begin{align}\label{Cartan tensor 2}
A_{ijk}=v\left(h_{ij}Y_k+h_{jk}Y_i+h_{ki}Y_j\right)+\frac{u-(n+1)v}{\|Y^\sharp\|^{2}}Y_{i}Y_{j}Y_{k}.
\end{align}
The terms $u$, $v$ and $Y^\sharp$ are defined as follows:
\begin{align*}
  &Y^\sharp:=Y_idy^i:=(b_i-s\alpha_{y^i})dy^i,\quad v:=\frac{\phi}{2}\left\{\log\left[\phi(\phi-s\phi')\right]\right\}',\\
  &u:=\frac{\phi}{2}\left\{\log\left[\phi^{n+1}(\phi-s\phi')^{n-2}\left[(\phi-s\phi')+(b^2-s^2)\phi''\right]\right]\right\}',
\end{align*}
where $\alpha_{y^i}=|y|^{-1}\delta_{ij}y^j$. Furthermore, if $v(s)=0$ on an open interval $I$, then $u(s)=0$ on $I$. Thus the set $\{s|v(s)=0, u(s)\neq0\}$ has empty interior.
\end{lem}
\begin{proof}
One has
\begin{align*}
&F_{y^iy^jy^k}=-\frac{1}{\alpha^2}(\phi-s\phi')\left[\alpha_{y^i}(\delta_{jk}-\alpha_{y^j}\alpha_{y^k})+\alpha_{y^j}(\delta_{ki}-\alpha_{y^k}\alpha_{y^i})+\alpha_{y^k}(\delta_{ij}-\alpha_{y^i}\alpha_{y^j})\right]\\
&-\frac{1}{\alpha^2}s\phi''\left[(b_i-s\alpha_{y^i})(\delta_{jk}-\alpha_{y^j}\alpha_{y^k})+(b_j-s\alpha_{y^j})(\delta_{ki}-\alpha_{y^k}\alpha_{y^i})+(b_k-s\alpha_{y^k})(\delta_{ij}-\alpha_{y^i}\alpha_{y^j})\right]\\
&-\frac{1}{\alpha^2}\phi''\left[\alpha_{y^i}(b_j-s\alpha_{y^j})(b_k-s\alpha_{y^k})+\alpha_{y^j}(b_k-s\alpha_{y^k})(b_i-s\alpha_{y^i})+\alpha_{y^k}(b_i-s\alpha_{y^i})(b_j-s\alpha_{y^j})\right]\\
&+\frac{1}{\alpha^2}\phi'''(b_i-s\alpha_{y^i})(b_j-s\alpha_{y^j})(b_k-s\alpha_{y^k}).
\end{align*}
By definition, we have
\begin{align*}
h_{ij}=FF_{y^iy^j}=\alpha\phi F_{y^iy^j}=\phi(\phi-s\phi')(\delta_{ij}-\alpha_{y^i}\alpha_{y^j})+\phi\phi''(b_i-s\alpha_{y^i})(b_j-s\alpha_{y^j})
\end{align*}
and
\begin{align*}
&A_{ijk} = \frac{F}{2}\left(F_{y^i}F_{y^jy^k}+F_{y^j}F_{y^ky^i}+F_{y^k}F_{y^iy^j}+FF_{y^iy^jy^k}\right)\\
         =&\frac{\phi}{2}\left[\phi'(\phi-s\phi')-s\phi\phi''\right]\left[(b_i-s\alpha_{y^i})(\delta_{jk}-\alpha_{y^j}\alpha_{y^k})+(b_j-s\alpha_{y^j})(\delta_{ki}-\alpha_{y^k}\alpha_{y^i})\right.\\
         &\left.+(b_k-s\alpha_{y^k})(\delta_{ij}-\alpha_{y^i}\alpha_{y^j})\right]+\frac{\phi}{2}\left(3\phi'\phi''+\phi\phi'''\right)(b_i-s\alpha_{y^i})(b_j-s\alpha_{y^j})(b_k-s\alpha_{y^k}).
\end{align*}

The one form $Y^\sharp:=(b_i-s\alpha_{y^i})dy^i$ defined on $V_0$ is essential. With respect to $\boldsymbol{\alpha}$, its dual vector field $Y$ at $y\in V\setminus\{0\}$ is just the projection of the dual of $\boldsymbol{\beta}$ on the orthogonal complement of $y$. It is clear that $Y^\sharp(y)=0$ if and only if $s=b$.

By Lemma \ref{Lemma 1}, we know that $\phi(\phi-s\phi')>0$ if $n\geq 3$. Then we have
\begin{align}\label{relation ang met}
\delta_{ij}-\alpha_{y^i}\alpha_{y^j}=\frac{1}{\phi(\phi-s\phi')}\left[h_{ij}-\phi\phi''Y_iY_j\right].
\end{align}
Using (\ref{relation ang met}), the Cartan tensor reduces to
\begin{align}\label{Cartan tensor 1}
 \begin{aligned}A_{ijk}=& \frac{\phi}{2}\left\{\log\left[\phi(\phi-s\phi')\right]\right\}'\left(h_{ij}Y_k+h_{jk}Y_i+h_{ki}Y_j\right)\\
         &+\frac{\phi}{2}\left\{3\phi'\phi''+\phi\phi'''-3\phi\phi''\left\{\log\left[\phi(\phi-s\phi')\right]\right\}'\right\}Y_{i}Y_{j}Y_{k}\\
        =:& v\left(h_{ij}Y_k+h_{jk}Y_i+h_{ki}Y_j\right)+w(Y_{i}Y_{j}Y_{k}),\\
        \end{aligned}
\end{align}
where we have denoted
\begin{align*}
v:=\frac{\phi}{2}\left\{\log\left[\phi(\phi-s\phi')\right]\right\}',\quad
w:=\frac{\phi}{2}\left\{3\phi'\phi''+\phi\phi'''-3\phi\phi''\left\{\log\left[\phi(\phi-s\phi')\right]\right\}'\right\}.
\end{align*}
Then the Cartan form is given by
\begin{equation}\label{eta 2}
A_i=g^{jk}A_{ijk}=\left((n+1)v+w\|Y^\sharp\|^2\right)Y_i.
\end{equation}

From (\ref{eta}) and (\ref{det}), the Cartan form has another expression
\begin{align}\label{eta 3}
A_{i}=\frac{\phi}{2}\left\{\log\left[\phi^{n+1}(\phi-s\phi')^{n-2}\left[(\phi-s\phi')+(b^2-s^2)\phi''\right]\right]\right\}'Y_{i}=:uY_i.
\end{align}
where we have denoted
\begin{align*}
u:=\frac{\phi}{2}\left\{\log\left[\phi^{n+1}(\phi-s\phi')^{n-2}\left[(\phi-s\phi')+(b^2-s^2)\phi''\right]\right]\right\}'.
\end{align*}
By (\ref{eta 2}) and (\ref{eta 3}), on the points where $Y^\sharp\neq0$ we have
\begin{align}\label{w}
w=[u-(n+1)v]\|Y^\sharp\|^{-2}.
\end{align}
Plugging (\ref{w}) into (\ref{Cartan tensor 1}), we obtain (\ref{Cartan tensor 2}).

Suppose that $v\equiv 0$ on a neighborhood of a point $s_0\in [-b,b]$. By Lemma \ref{Lemma 5} in Appendix, $\phi^2$ is a quadratic function of $s$ and  $u=0$ on this neighborhood.  Hence the area on which
\begin{align*}
A_{ijk}=\frac{u}{\|Y^\sharp\|^2}Y_iY_jY_k\neq0
\end{align*}
has no interior points.
\end{proof}

Setting
${n+q-1}:={u}{v}^{-1}$
on the set $v\neq0$, the Cartan tensor is reformulated as
\begin{align}\label{Cartan tensor 3}
A_{ijk}=\frac{1}{n+q-1}\left(h_{ij}A_k+h_{jk}A_i+h_{ki}A_j+\frac{q-2}{\|\eta\|^{2}}A_{i}A_{j}A_{k}\right).
\end{align}
If we regard (\ref{Cartan tensor 2}) as the limit of (\ref{Cartan tensor 3}) when $q$ takes value $1-n$ or $\infty$, then (\ref{Cartan tensor 3}) holds for any $(\alpha,\beta)$-norms. We call $q=1-n+uv^{-1}$ the \emph{characteristic function} of the $(\alpha,\beta)$-norm (\ref{alpha beta norm}).

The function $q$ is inspired by some works on Lagrangian submanifolds \cite{Cas} and deserves further investigation.

The equation ${n+q-1}=u{v}^{-1}$ is in fact a third order non-linear ODE of $\phi$ for a given $q$. Some requisite properties of the solutions of the equation $q\equiv1$ will be presented in Appendix.

The follow lemma is essential for the next section.
\begin{lem}\label{theo 6}
Let $\mathbf{F}$ be an $(\alpha,\beta)$-norm defined on $V$ of dimension $n\geq4$. On a connected cone $W\subset V$, if $v\neq0$ and the characteristic function $q\equiv1$, then $\|\hat{T}\|$ is constant on $\mathbf{I_F}\cap W$.
\end{lem}
\begin{proof}
By (\ref{Cartan tensor 3}) the cubic form $\hat{C}$ and the Tchebychev form $\hat{T}$ of the indicatrix $\mathbf{I_F}\cap W$ satisfy
\begin{align}\label{q=1 c reducible}
C_{\alpha\beta\gamma}=\frac{n-1}{n}\left[\left(\delta_{\alpha\beta}T_{\gamma}+\delta_{\beta\gamma}T_{\alpha}+\delta_{\gamma\alpha}T_{\beta}\right)-\frac{1}{\|\hat{T}\|^2}T_{\alpha}T_{\beta}T_{\gamma}\right]
\end{align}
with respect to a local orthonormal frame of the affine metric $\mathsf{h}$. From (\ref{Rh intergral condition c}), the curvature tensor of $\mathsf{h}$ is given by
\begin{align}\label{Riemann curvatue h}
R_{\beta\gamma\mu\nu}=-(\delta_{\beta\mu}\delta_{\gamma\nu}-\delta_{\beta\nu}\delta_{\gamma\mu})+\sum_{\alpha}(C_{\alpha\beta\mu}C_{\alpha\gamma\nu}-C_{\alpha\beta\nu}C_{\alpha\gamma\mu}).
\end{align}
Plugging (\ref{q=1 c reducible}) into (\ref{Riemann curvatue h}) we get
\begin{align}\label{Riemann curvature p=1}
R_{\beta\gamma\mu\nu}=-\left(1-\frac{(n-1)^2}{n^2}\|\hat{T}\|^2\right)(\delta_{\beta\mu}\delta_{\gamma\nu}-\delta_{\beta\nu}\delta_{\gamma\mu}).
\end{align}
So the metric $\mathsf{h}$ has isotropic sectional curvature. According to the assumption on the dimension, Schur's lemma implies that $\|\hat{T}\|^2$ is constant on $\mathbf{I_F}\cap W$.
\end{proof}

\begin{remark}
  The assumption of dimension $n\geq4$ in Lemma \ref{theo 6} guarantee that the Schur's lemma can be applied on the indicatrix of the $(\alpha,\beta)$-norm. In the situation of $n=2$ or $n=3$, the method used here is invalid.
 We will leave the problem in the remain cases for the future study.
\end{remark}

\section{Equivalence theorems of the $(\alpha,\beta)$-norms}
In this section, we state and prove an equivalence theorem of $(\alpha,\beta)$-norms.
\begin{theo}\label{theo c reducible equivalent}
Let $(V,\mathbf{F})$ and $(\tilde{V},\tilde{\mathbf{F}})$ be $n\geq4$ dimensional vector spaces with $(\alpha,\beta)$-norms defined on $V$ and $\tilde{V}$, respectively.

Let $f:V\rightarrow \tilde{V}$ be a norm preserving map which is diffeomorphism between $V_0$ and $\tilde{V}_0$, and satisfies
\begin{align*}
f(ty)=tf(y), \quad \forall y\in V_0,~t>0.
\end{align*}
Assume that $\bar{g}=f^*\tilde{\bar{g}}$.   Then we have
\begin{enumerate}
  \item[(a)] On the set $v\neq0$, $q=1$ if and only if $\tilde{q}=1$. We then can define the following set
  \begin{equation*}
  Q:=\{y\in V_0|v(y)\neq0,q(y)=1\}=\{y\in V_0|f^*\tilde{v}\neq0, (f^*\tilde{q})(y)=1\}.
\end{equation*}
  \item[(b)] $Q^c=V_0\setminus Q\neq\emptyset$ is an open subset.  Furthermore $A=\pm f^*\tilde{A}$ holds on $Q^c$.
\end{enumerate}

\end{theo}

\begin{proof}
On the hypersurfaces $\mathbf{I_F}$ and $\mathbf{I}_{\tilde{\mathbf{F}}}$, we have the centroaffine geometric structures induced by the identity maps.

Because that $\bar{g}=f^*\tilde{\bar{g}}$, the induced affine metrics satisfy $\mathsf{h}=f^*\tilde{\mathsf{h}}$. During the proof, the notations of the geometric invariants on $\mathbf{I}_{\tilde{\mathbf{F}}}$ pulled back by $f$  will omit the symbol $f^*$ for convenient.

The Riemann curvature tensors of $\mathsf{h}$ and $\tilde{\mathsf{h}}$ coincide.
Let $\mathbf{e}_1,\cdots,\mathbf{e}_{n-1}$ be an arbitrary local orthonormal frame field of the metric $\mathsf{h}$ around a point $y\in \mathbf{I}_{\mathbf{F}}$.
From (\ref{Riemann curvatue h}), we have
\begin{align}\label{S equation 1}
\sum_{\alpha}(C_{\alpha\beta\mu}C_{\alpha\gamma\nu}-C_{\alpha\beta\nu}C_{\alpha\gamma\mu})=\sum_{\alpha}(\tilde{C}_{\alpha\beta\mu}\tilde{C}_{\alpha\gamma\nu}-\tilde{C}_{\alpha\beta\nu}\tilde{C}_{\alpha\gamma\mu}),
\end{align}
where $\hat{C}$ and $\tilde{\hat{C}}$ are the cubic forms of $\mathbf{I_F}$ and $\mathbf{I_{\tilde{F}}}$, respectively.
Consequently, from (\ref{S equation 1}) we have
\begin{align}\label{S equation 2}
(n-1)\sum_{\alpha}T_{\alpha}C_{\alpha\gamma\nu}-\sum_{\alpha,\beta}C_{\alpha\beta\nu}C_{\alpha\gamma\beta}=(n-1)\sum_{\alpha}\tilde{T}_{\alpha}\tilde{C}_{\alpha\gamma\nu}-\sum_{\alpha,\beta}\tilde{C}_{\alpha\beta\nu}\tilde{C}_{\alpha\gamma\beta},
\end{align}
and
\begin{align}\label{S equation 3}
(n-1)^2\|\hat{T}\|^2-\|\hat{C}\|^2=(n-1)^2\|\tilde{\hat{T}}\|^2-\|\tilde{\hat{C}}\|^2,
\end{align}
where $\hat{T}$ and $\tilde{\hat{T}}$ are the Tchebychev forms of $\mathbf{I_F}$ and $\mathbf{I_{\tilde{F}}}$, respectively.

By the assumption and (\ref{Cartan tensor 2}), we obtain
\begin{align}\label{Mq=0}
\begin{aligned}
&C_{\alpha\beta\gamma}=v\left(\delta_{\alpha\beta}\mathsf{Y}_{\gamma}+\delta_{\beta\gamma}\mathsf{Y}_{\alpha}+\delta_{\gamma\alpha}\mathsf{Y}_{\beta}\right)+\frac{u-(n+1)v}{\|\mathsf{Y}^\sharp\|^2}\mathsf{Y}_{\alpha}\mathsf{Y}_{\beta}\mathsf{Y}_{\gamma},\\
&\tilde{C}_{\alpha\beta\gamma}=\tilde{v}\left(\delta_{\alpha\beta}\tilde{\mathsf{Y}}_{\gamma}+\delta_{\beta\gamma}\tilde{\mathsf{Y}}_{\alpha}+\delta_{\gamma\alpha}\tilde{\mathsf{Y}}_{\beta}\right)+\frac{\tilde{u}-(n+1)\tilde{v}}{\|\tilde{\mathsf{Y}}^\sharp\|^2}\tilde{\mathsf{Y}}_{\alpha}\tilde{\mathsf{Y}}_{\beta}\tilde{\mathsf{Y}}_{\gamma},
\end{aligned}
\end{align}
where we denote ${\mathsf{Y}^\sharp}:=i^*_{\mathbf{I_F}}Y^\sharp$, and $i_{\mathbf{I_F}}$ is the inclusion map of $\mathbf{I_F}$ in $V$. $\tilde{\mathsf{Y}}^\sharp$ is defined similarly.

 From  (\ref{Mq=0}), a direct calculation gives
\begin{align}\label{Cnorm=pTnorm}
\begin{aligned}&\|\hat{C}\|^2=(n-1)^2[3(n-2)v^2+(u+(2-n)v)^2]\|\mathsf{Y}^\sharp\|^2,\\
&\|\tilde{\hat{C}}\|^2=(n-1)^2[3(n-2)\tilde{v}^2+(\tilde{u}+(2-n)\tilde{v})^2]\|\tilde{\mathsf{Y}}^\sharp\|^2.
\end{aligned}
\end{align}
Plugging (\ref{Cnorm=pTnorm}) into (\ref{S equation 3}), we find
\begin{align}\label{key1}
(2u-(n+1)v)v\|\mathsf{Y}^\sharp\|^2=(2\tilde{u}-(n+1)\tilde{v})\tilde{v}\|\tilde{\mathsf{Y}}^\sharp \|^2.
\end{align}
Similarly plugging  (\ref{Mq=0}) in (\ref{S equation 2}), one has
\begin{align}\label{key2}
\begin{aligned}
&(n-3)(u-nv)v\mathsf{Y}_{\gamma}\mathsf{Y}_{\nu}+(u-2v)v\delta_{\gamma\nu}\|\mathsf{Y}^\sharp\|^2\\
=&(n-3)(\tilde{u}-n\tilde{v})\tilde{v}\tilde{\mathsf{Y}}_{\gamma}\tilde{\mathsf{Y}}_{\nu}+(\tilde{u}-2\tilde{v})\tilde{v}\delta_{\gamma\nu}\|\tilde{\mathsf{Y}}^\sharp\|^2.
\end{aligned}
\end{align}
As $n>3$, choosing indices $\gamma\neq \nu$ in (\ref{key2}) gives
\begin{align}\label{key3}
(u-nv)v\mathsf{Y}_{\gamma}\mathsf{Y}_{\nu}=(\tilde{u}-n\tilde{v})\tilde{v}\tilde{\mathsf{Y}}_{\gamma}\tilde{\mathsf{Y}}_{\nu}.
\end{align}

Set
\begin{align}\label{U}
U=\{\mathsf{Y}^\sharp=0\}\cup\{\tilde{\mathsf{Y}}^\sharp=0\}\cup \partial U_0\cup \partial \tilde{U}_0,
\end{align}
where $U_0=\{y\in \mathbf{I}_{\mathbf{F}}|v(y)=0\}$ and $\tilde{U}_0=\{y\in \mathbf{I}_{\mathbf{F}}|\tilde{v}(y)=0\}$.
By Lemma \ref{Lemma 4}, one has
\begin{equation}\label{strange inclusion}
  \{y|v(y)=0~{\rm and}~u(y)\neq0\}\subset\partial U_0, \quad \{y|\tilde{v}(y)=0~{\rm and}~\tilde{u}(y)\neq0\}\subset\partial \tilde{U}_0.
\end{equation}
The complement $U^c$ is open and dense in $\mathbf{I_F}$. We only need to prove the theorem on the set $U^c$.

\vspace{0.3cm}

\noindent\textbf{Case 1:}  Assume that $\tilde{v}(y)=0$.
\vspace{0.2cm}

Thus there is a neighborhood $U_y\subset U^c$ of $y$ such that $\tilde{v}=0$ on $U_y$.
By Lemma \ref{Lemma 4}, $\tilde{\mathbf{F}}$ is Euclidean on $U_y$.  By (\ref{strange inclusion}) and the definition of $U^c$, one has $\tilde{u}(y)=0$.
As $\mathsf{Y}^\sharp(y)\neq0$, we can change the frame field $\mathbf{e}_1,\ldots,\mathbf{e}_{n-1}$ such that $\mathsf{Y}_\gamma(y)\neq0$ for $\gamma=1,\ldots,n-1$.
By (\ref{key3}), it is clear that $(u(y)-nv(y))v(y)=0$. We claim that $v(y)=0$. Therefore $u(y)=0$ and $\hat{C}(y)=\tilde{\hat{C}}(y)=0$.

Otherwise $v(y)\neq0$. Therefore $q\equiv1$ on a smaller neighborhood of $y$ which we will also denote by $U_y$. From Lemma \ref{theo 6}, we know that $\|\hat{T}\|$ is a constant $U_y$. Since $\mathsf{h}=\tilde{\mathsf{h}}$ and $\tilde{\mathbf{F}}$ is Euclidean on $U_y$, $\mathsf{h}$ has constant curvature 1. Therefore $\|\hat{T}\|=0$. Since $\hat{T}=u\mathsf{Y}^\sharp$ and $\mathsf{Y}^\sharp\neq0$, we obtain $u=0$ on $U_y$. By (\ref{key3}) and $\mathsf{Y}^\sharp\neq0$ again, one has $v(y)=0$. It is a contradiction.

\vspace{0.3cm}

\noindent\textbf{Case 2:}  Assume that $\tilde{v}(y)\neq0$. According to Case 1, $v(y)\neq0$.
\vspace{0.2cm}

Set $Z=v\mathsf{Y}^\sharp$ and $\tilde{Z}=\tilde{v}\tilde{\mathsf{Y}}^\sharp$. We rewrite (\ref{key3}) at the point $y$ as following
\begin{align}\label{key3'}
(q-1)Z_{\gamma}Z_{\nu}=(\tilde{q}-1)\tilde{Z}_{\gamma}\tilde{Z}_{\nu}.
\end{align}

As $Z\neq0$ and $\tilde{Z}\neq0$, (\ref{key3'}) implies that $q=1$ if and only if $\tilde{q}=1$. Hence the second assertion (a) is valid. Furthermore, $Q\neq V_0$ by Lemma \ref{Lemma 5} in Appendix.

Now we choose the frame field $\mathbf{e}_1,\ldots,\mathbf{e}_{n-1}$ such that $\tilde{Z}_\gamma\neq0$ for $\gamma=1,\ldots,n-1$.
Under the assumption on dimension $n>3$, for each index $\gamma$ we can choose indices $\mu$ and $\nu$ which are mutually different. At $y$,  one has
\begin{align}\label{key3''}
(q-1)Z_{\gamma}Z_{\mu}=(\tilde{q}-1)\tilde{Z}_{\gamma}\tilde{Z}_{\mu},
\end{align}
and
\begin{align}\label{key3'''}
(q-1)Z_{\nu}Z_{\mu}=(\tilde{q}-1)\tilde{Z}_{\nu}\tilde{Z}_{\mu}.
\end{align}
From (\ref{key3'})-(\ref{key3'''}), one gets
\begin{align}\label{key4}
(q-1)Z_{\gamma}^2=(\tilde{q}-1)\tilde{Z}_{\gamma}^2, \quad \gamma=1,\ldots,n-1.
\end{align}
Hence
\begin{align}\label{key4'}
(q-1)\|Z\|^2=(\tilde{q}-1)\|\tilde{Z}\|^2.
\end{align}
We rewrite (\ref{key1}) as
\begin{equation}\label{key1'}
  (2q+n-3)\|Z\|^2=(2\tilde{q}+n-3)\|\tilde{Z}\|^2.
\end{equation}
By (\ref{key1'}) and (\ref{key4'}), we obtain $q=\tilde{q}$ on the set $q\neq1$.

As a consequence, we obtain form (\ref{key3'}) and (\ref{key4}) the following equations on the set $q\neq1$

\begin{align}\label{key4''}
&Z_{\gamma}Z_{\mu}=\tilde{Z}_{\gamma}\tilde{Z}_{\mu}, \quad \gamma\neq \mu,\\
&Z_{\gamma}^2=\tilde{Z}_{\gamma}^2, \quad \gamma=1,\ldots,n-1.
\end{align}
Therefore $Z_{\gamma}=\pm\tilde{Z}_{\gamma}$. Because that $\tilde{Z}_{\gamma}$ are not zero, from (\ref{key4''}) we obtain
\begin{equation*}
  Z=\pm\tilde{Z}.
\end{equation*}
On the set $v\neq0$, (\ref{Mq=0}) can be written as
\begin{align}\label{Mq=0'}
\begin{aligned}
&C_{\alpha\beta\gamma}=\left(\delta_{\alpha\beta}Z_{\gamma}+\delta_{\beta\gamma}Z_{\alpha}+\delta_{\gamma\alpha}Z_{\beta}\right)+\frac{(q-2)}{\|Z\|^2}Z_{\alpha}Z_{\beta}Z_{\gamma},\\
&\tilde{C}_{\alpha\beta\gamma}=\left(\delta_{\alpha\beta}\tilde{Z}_{\gamma}+\delta_{\beta\gamma}\tilde{Z}_{\alpha}+\delta_{\gamma\alpha}\tilde{Z}_{\beta}\right)+\frac{\tilde{q}-2}{\|\tilde{Z}\|^2}\tilde{Z}_{\alpha}\tilde{Z}_{\beta}\tilde{Z}_{\gamma}.
\end{aligned}
\end{align}
On the set $q\neq1$, by $q=\tilde{q}$ and  $Z=\pm\tilde{Z}$, we obtain $C=\pm\tilde{C}$.

In summery,  we obtain that $A=\pm\tilde{A}$ holds on $Q^c$, as the tensor $A$ is zero on directions along the rays emitted from the origin.
\end{proof}

\begin{remark}
\textnormal{ The Legendre transformation from a Minkowski space to it dual space with the dual Minkowski norm preserves the fundamental forms but changes the sign of the Cartan tensors.
 Hence there are Minkowski spaces with the same metric $\bar{g}$, but different Cartan tensors. }
\end{remark}

\section{The proof of Theorem 1}
In this section we will prove that Landsberg spaces with general $(\alpha,\beta)$-metrics are Berwald type. The strategy is using the equivalence theorem in Section 3 and the nonlinear parallel transport.

First, we would like to discuss a structure of the tangent bundles of Finsler manifolds with general $(\alpha,\beta)$-metrics.

Let $V$ be a real vector space of dimension $n\geq3$. Let $\mathbf{F}=\boldsymbol{\alpha}\phi(s)=\boldsymbol{\alpha}\phi\left(\frac{\boldsymbol{\beta}}{\boldsymbol{\alpha}}\right)$ be an $(\alpha,\beta)$-norm on $V$ with $b=\|\boldsymbol{\beta}\|_{\boldsymbol{\alpha}}$, where $\phi$ is a smooth positive function defined on $I$ such that $[-b,b]\subset I$ and $\phi$ satisfies (\ref{conditions phi form norm}) on $[-b,b]$.

We will divide $V_0$ into two parts. Each part is a cone. By Lemma \ref{Lemma 4} and Lemma \ref{Lemma 5}, on the interior of $\{y\in V_0|v(s)=0\}$, the norm $\mathbf{F}$ is Euclidean. On the set $\{y\in V_0|v(s)\neq0\}$, the characteristic function $q=uv^{-1}-n+1$ is well defined. We will denote the closed cone $\{y\in V_0|v(s)\neq0, q(s)\equiv1\}$ by $Q$. Therefore the compliment $Q^c$ of $Q$ is open.

By Lemma \ref{Lemma 5}, $Q^c$ is nonempty. Assume that $Q$ has nonempty interior. Then each connected subset of $\{s(y)|y\in Q\}$ is a closed subinterval of $(-b,b)$. Let $I$ be one of nontrivial closed intervals. Then $Q_I:=\{y\in Q|s(y)\in I\}$ is a closed subcone of $Q$ and $Q_I\cap\overline{Q^c}\neq\emptyset$.

According to Lemma \ref{Lemma 5}, $\phi$ coincides with the following function on $I$
\begin{equation}\label{solution q=1'}
   \tilde{\phi}=c_2\exp\int\frac{(1-c_0)c_1s+\sqrt{b^2-s^2}}{c_0c_1+s((1-c_0)c_1s+\sqrt{b^2-s^2})} ds,
  \end{equation}
where $c_0>0$, $c_1\neq0$ and $c_2>0$ are constants. Therefore there is a larger open interval $\tilde{I}$ which contains $I$, such that the function $\tilde{\phi}$ in (\ref{solution q=1'}) defines on $\tilde{I}$ and still satisfies (\ref{conditions phi form norm}). Set $Q_{\tilde{I}}:=\{y\in Q|s(y)\in\tilde{I}\}$. From the analyticity of $\tilde{\phi}$, the functional $\tilde{\mathbf{F}}=\boldsymbol{\alpha}\tilde{\phi}\left(\frac{\boldsymbol{\beta}}{\boldsymbol{\alpha}}\right)$ on
$Q_{\tilde{I}}$ is analytic and satisfies $\tilde{\mathbf{F}}|_{Q_I}=\mathbf{F}|_{Q_I}$. Since $\tilde{\mathbf{F}}$ and $\mathbf{F}$ are both smooth, on $Q_I\cap\overline{Q^c}\neq\emptyset$, we have
\begin{equation*}
  \frac{\partial^{|\alpha|}\tilde{F}}{\partial y^{\alpha_1}\cdots\partial y^{\alpha_n}}= \frac{\partial^{|\alpha|}F}{\partial y^{\alpha_1}\cdots\partial y^{\alpha_n}},
\end{equation*}
for any multiindix $\alpha=(\alpha_1,\ldots,\alpha_n)$.

Let $M$ be a smooth manifold with a general $(\alpha,\beta)$-metric $\mathbf{F}=\boldsymbol{\alpha}\phi\left(x,\frac{\boldsymbol{\beta}}{\boldsymbol{\alpha}}\right)$, where $\boldsymbol{\alpha}$ is a Riemannian metric, $\boldsymbol{\beta}$ an 1-form with length $b=\|\boldsymbol{\beta}\|_{\boldsymbol{\alpha}}$ and $\phi(x,s)$ a smooth positive function on $\{(x,s)\in M\times\mathbb{R}|b(x)\geq s\geq -b(x), x\in M\}$. By the above discussion of a single Minkowski space with an $(\alpha,\beta)$-norm, the slit tangent bundle splits $TM_0=\mathcal{Q}\cup \mathcal{Q}^c$, where $\mathcal{Q}=\{q(x,s)\equiv1\}$ and $\mathcal{Q}^c\neq\emptyset$. For simplicity, we assume $\mathcal{Q}$ is connected and has nonempty interior without loss of generality. There is a larger open tangent cone $\tilde{\mathcal{Q}}$ containing  $\mathcal{Q}$ and a function on $$\{(x,s(y))\in M\times[-b,b]|y\in\tilde{\mathcal{Q}}_x, (x,y)\in TM_0\}$$ of the following form
\begin{equation}\label{solution q=1''}
   \tilde{\phi}(x,s)=c_2(x)\exp\int\frac{(1-c_0(x))c_1(x)s+\sqrt{b^2(x)-s^2}}{c_0(x)c_1(x)+s((1-c_0(x))c_1(x)s+\sqrt{b^2(x)-s^2})} ds,
  \end{equation}
such that $\tilde{\mathbf{F}}=\boldsymbol{\alpha}\tilde{\phi}\left(x,\frac{\boldsymbol{\beta}}{\boldsymbol{\alpha}}\right)$ coincides with $\mathbf{F}$ on $\mathcal{Q}$, where $c_0(x)$, $c_1(x)$, $c_2(x)$ and $b(x)$ are smooth functions on $M$.

Since $\phi(x,s)$ is analytic in the variable $s$ when $x$ is fixed, then $\tilde{\mathbf{F}}$ is analytic in $y$ variables. Furthermore from the formula (\ref{solution q=1''}), the partial derivatives $\frac{\partial\phi}{\partial x^i}$, $i=1,\ldots,n$, are elementary functions and also analytic in the variable $s$. Thus $\frac{\partial\tilde{F}}{\partial x^i}$ are analytic in $y$ variables and on $\mathcal{Q}\cap\overline{\mathcal{Q}^c}$
\begin{equation*}
  \frac{\partial^{|\alpha|+1}\tilde{F}}{\partial x^i\partial y^{\alpha_1}\cdots\partial y^{\alpha_n}}= \frac{\partial^{|\alpha|+1}F}{\partial x^i\partial y^{\alpha_1}\cdots\partial y^{\alpha_n}},
\end{equation*}
for $i=1,\ldots,n$, and any multiindix $\alpha=(\alpha_1,\ldots,\alpha_n)$.

\begin{proof}[Proof of Theorem 1]

Assume that $(M,\mathbf{F})$ is a Landsberg manifold with a general $(\alpha,\beta)$-metric $$\mathbf{F}=\boldsymbol{\alpha}\phi\left(x,\frac{\boldsymbol{\beta}}{\boldsymbol{\alpha}}\right),$$ where $\boldsymbol{\alpha}$ is the a Riemannian metric, $\boldsymbol{\beta}$ is a 1-form with length $b:=\|\boldsymbol{\beta}\|_{\boldsymbol{\alpha}}$, $\phi$ is a smooth function on a neighborhood of $\{(x,s)\in M\times\mathbb{R}|b(x)\geq s\geq -b(x), x\in M\}.$
For $x\in M$, the restrictions of $\mathbf{F}$ on $T_xM$ is an $(\alpha,\beta)$-norm.
Let $$Q_x=\{y\in T_xM\setminus\{0\}|v(x,s)\neq0, q(x,s)=1\},$$ then $\mathcal{Q}:=\bigcup_{x\in M}Q_x$ is a closed in $TM_0$. The complement $\mathcal{Q}^c$ of $\mathcal{Q}$ in $TM_0$ is nonempty and open.

Let $\sigma:[0,1]\rightarrow M$ be a smooth curve emanates form $\sigma(0)=p\in M$ such that the parallel transports are defined along $\sigma$. Then $\bar{g}$ is preserved by the parallel transports along $\sigma$. Therefore
$P_{\sigma,t}^*\bar{g}_{\sigma(t)}=\bar{g}_{p}$, for $\forall t\in[0,1]$.

By Lemma \ref{lemmma basic pro of npt} and Theorem \ref{theo c reducible equivalent}, $P_{\sigma,t}$ is linear on $Q_p^c$ and maps $Q_p$ to $Q_{\sigma(t)}$, for $t\in(0,1]$.
Therefore, the spray coefficient $G^i$ is quadratic on the open tangent cone $\mathcal{Q}^c$, $i=1,\ldots,n$.

From the formula (\ref{G}), for $i=1,\ldots,n$, $G^i$ involves  $\mathbf{F}$ and its first order derivatives of the variables $x$, and other elementary functions of $y$ variables. There is an analytic function $\tilde{G}^i$ of variables $y$ on on a larger open tangent cone $\tilde{\mathcal{Q}}$ containing $\mathcal{Q}$, such that $\tilde{G}^i|_{\mathcal{Q}}=G^i|_{\mathcal{Q}}$, $i=1,\ldots,n$. On $\mathcal{Q}\cap\overline{\mathcal{Q}^c}$, we have
\begin{equation*}
  \frac{\partial^3G^i}{\partial y^j\partial y^k\partial y^l}=\frac{\partial^3\tilde{G}^i}{\partial y^j\partial y^k\partial y^l},\quad i,j,k,l=1,\ldots,n.
\end{equation*}

For any $i=1,\ldots,n$, since $G^i$ is quadratic on $\mathcal{Q}^c\neq\emptyset$, we have
$
   \frac{\partial^3G^i}{\partial y^j\partial y^k\partial y^l}=0
$
on $\mathcal{Q}\cap\overline{\mathcal{Q}^c}$. Therefore $\tilde{G}^i$ is quadratic on $\tilde{\mathcal{Q}}$ and $G^i$ is quadratic on $TM_0$. The proof is finished.
\end{proof}

\section*{Appendix}
In the appendix, we mainly investigate the ODE: $ u=nv$,
where $$v:=\frac{\phi}{2}\left\{\log\left[\phi(\phi-s\phi')\right]\right\}', u:=\frac{\phi}{2}\left\{\log\left[\phi^{n+1}(\phi-s\phi')^{n-2}\left[(\phi-s\phi')+(b^2-s^2)\phi''\right]\right]\right\}',$$
  and $\phi=\phi(s)$ is a positive smooth function satisfies (\ref{conditions phi form norm}).

Without loss of generality, we will assume $b=1$ in this appendix.

\begin{lem}\label{Lemma 5}
Let $\phi$ be a positive solution of the equation $u=nv$.  Then $\phi$ has only two possibilities:
\begin{enumerate}
  \item[(a)] $\phi=\sqrt{c_0+c_1s^2}$, where $c_0>0$ and $c_1$ are constants. In this case, $v\equiv0$ and $u\equiv0$.  Let $I_1$ be the set on which $\phi$ satisfies the condition (\ref{conditions phi form norm}). Then $I_1\neq\emptyset$ if $c_0$ and $c_1$ are chosen suitably.
  \item[(b)] $\phi$ has the following form
  \begin{equation}\label{solution q=1}
    \phi=c_2\exp\int\frac{(1-c_0)c_1s+\sqrt{1-s^2}}{c_0c_1+s((1-c_0)c_1s+\sqrt{1-s^2})} ds,
  \end{equation}
  where $c_0>0$, $c_1\neq0$ and $c_2>0$ are constants.
    Let $I$ be one of the maximal connected interval of $\phi$, then $I\subseteq(-1,1)$. $\phi$ is analytic on $I$ and can not be extend to positive smooth functions across their singularises. Moreover $\phi$  satisfies the inequalities (\ref{conditions phi form norm}) and the function $v$ is nowhere zero on $I$.

\end{enumerate}
\end{lem}

\begin{proof}
Suppose that $v\equiv 0$. we will prove that $\phi^2$ is a quadratic function of $s$. In fact, $v\equiv 0$ implies that there is a constant $c_0>0$ and
\begin{align*}
\phi(\phi-s\phi')=c_0.
\end{align*}
The general solutions are
\begin{align}\label{sqrt quadraic}
\phi=\sqrt{c_0+c_1s^2},
\end{align}
where $c_1$ is a constant.  It is easy to verify $u\equiv0$ and other claims of (a) for $\phi=\sqrt{c_0+c_1s^2}$.

Next we will solve the equation $nv=u$ which has the following explicit form
\begin{align}\label{ODE q=1}
n\left[\log \phi(\phi-s\phi')\right]'=\left\{\log \left[\phi^{n+1}(\phi-s\phi')^{n-2}\left[(\phi-s\phi')+(1-s^2)\phi''\right]\right]\right\}'
\end{align}
except from the trivial solutions in the form (\ref{sqrt quadraic}).

Taking integration of both sides of (\ref{ODE q=1}), we get
\begin{align*}
n\log \phi(\phi-s\phi')=\log\left[ \phi^{n+1}(\phi-s\phi')^{n-2}\left[(\phi-s\phi')+(1-s^2)\phi''\right]\right]+c,
\end{align*}
or equivalently
\begin{align}\label{ODE q=1 2}
(\phi-s\phi')^{2}=c_0\phi\left[(\phi-s\phi')+(1-s^2)\phi''\right],
\end{align}
where $c$ is a constant and $c_0=e^{c}>0$. It is remarkable that (\ref{ODE q=1 2}) is independent of the factor $n$.

Setting $\psi=(\log \phi)'$, the equation (\ref{ODE q=1 2}) changes to
\begin{align}\label{ODE q=1 3}
(1-s\psi)^2=c_0\left[1-s\psi+(1-s^2)(\psi'+\psi^2)\right].
\end{align}
On the interval $s\neq \pm 1$, (\ref{ODE q=1 3}) is a Riccati equation
\begin{align}\label{Riccati}
\psi'=\frac{(1+c_0^{-1})s^2-1}{1-s^2}\psi^2+\frac{(1-2c_0^{-1})s}{1-s^2}\psi+\frac{c_0^{-1}-1}{1-s^2}.
\end{align}
The equation (\ref{Riccati}) has analytic solutions on $(-1,1)$ except from some singularities in the following form
\begin{equation}\label{solution Riccati}
  \psi=\left(s+\frac{c_1c_0}{c_1(1-c_0)s+\sqrt{1-s^2}}\right)^{-1}=\frac{c_1(1-c_0)s+\sqrt{1-s^2}}{c_1c_0+s(c_1(1-c_0)s+\sqrt{1-s^2})},
\end{equation}
where $c_1$ is a constant.

It follows that the function $\phi(s)=\exp\int\psi(s)ds$ or more explicitly (\ref{solution q=1}) solves (\ref{ODE q=1}). $\phi$ is positive, and moreover
\begin{equation}\label{1st ineq holds}
  \frac{1}{\phi}(\phi-s\phi')=1-s(\log\phi)'=1-s\psi=\frac{c_1c_0}{c_1c_0+s(c_1(1-c_0)s+\sqrt{1-s^2})}
\end{equation}
on its domain. Then $\phi$ satisfies $\phi-s\phi'>0$ only if $c_1\neq0$. By (\ref{ODE q=1 2}) and (\ref{1st ineq holds}), $\phi$  satisfies (\ref{conditions phi form norm}) on some connected intervals of its domain.

By (\ref{1st ineq holds}), we also verifies that
\begin{equation*}
  [\phi(\phi-s\phi')]'=\frac{c_1c_0\phi^2}{\sqrt{1-s^2}[c_1c_0+s(c_1(1-c_0)s+\sqrt{1-s^2})]^2}.
\end{equation*}
Therefore $v$ is nowhere zero.

We will provide an elementary analysis of the set of the singularities of $\phi$ below.

\textbf{Case 1:} $c_0=1$. In this case

\begin{equation}\label{Riccati c=1 solution '}
  \psi=\frac{\sqrt{1-s^2}}{c_1+s\sqrt{1-s^2}}, \quad {\rm and}\quad
  \psi'=\frac{-\frac{c_1s}{\sqrt{1-s^2}}-1+s^2}{(c_1+s\sqrt{1-s^2})^2}.
\end{equation}

If $|c_1|>1/2$, then $\psi$ is a bounded function on $(-1,1)$. Thus $\phi=\exp\int\psi$ has positive lower bound on $(-1,1)$. By (\ref{Riccati c=1 solution '}), we obtain
\begin{equation*}
  \phi''=\phi\psi^2+\phi\psi'=O\left(\frac{1}{\sqrt{1\mp s}}\right), \quad s\to\pm 1.
\end{equation*}
Thus $\phi$ has $\pm1$ as its singularities.

If $c_1=\frac{1}{2}$,  $\psi$ has an extra singularity at $s=-\frac{\sqrt{2}}{2}$ except from $\pm1$. And
\begin{equation*}
  \psi=O\left(\frac{1}{(s+\frac{\sqrt{2}}{2})^2}\right), \quad
 \phi=\exp\int\psi ds\to +\infty, \quad s\to-\frac{\sqrt{2}}{2}
\end{equation*}
In this situation, the set of singularities of $\phi$ is $\{\pm1,-\frac{\sqrt{2}}{2}\}$.

If $c_1=-\frac{1}{2}$,  $\psi$ has a singularity at $s=\frac{\sqrt{2}}{2}$ except from $\pm1$. And
\begin{equation*}
  \psi=O\left(-\frac{1}{(s-\frac{\sqrt{2}}{2})^2}\right), \quad
 \phi=\exp\int\psi ds=O\left(\exp\left(-\frac{1}{\left|s-\frac{\sqrt{2}}{2}\right|}\right)\right)\to 0, \quad s\to\frac{\sqrt{2}}{2}
\end{equation*}
In this situation, the set of singularities of $\phi$ is $\{\pm1,\frac{\sqrt{2}}{2}\}$.

If $0<c_1<\frac{1}{2}$, $\psi$ has two more singularities at $s_1=\sqrt{\frac{1-\sqrt{1-4c_1^2}}{2}}$ and $s_2=\sqrt{\frac{1+\sqrt{1-4c_1^2}}{2}}$ except from $\pm1$.
We have
\begin{equation*}
  \psi=O\left(\frac{1}{s-s_1}\right), \quad s\to s_1
\end{equation*}
\begin{equation*}
  \int\psi=O\left(-\log(s-s_1)\right), \quad s\to s_1^+, \quad \int\psi=O\left(\log(s_1-s)\right), \quad s\to s_1^-,
\end{equation*}
and
\begin{equation*}
  \phi=\exp\int\psi=O\left(\frac{1}{s-s_1}\right), \quad s\to s_1^+, \quad \phi=\exp\int\psi=O\left(s_1-s\right), \quad s\to s_1^-.
\end{equation*}
By the same argument, we obtain
\begin{equation*}
  \phi=\exp\int\psi=O\left(\frac{1}{s_2-s}\right), \quad s\to s_2^-, \quad \phi=\exp\int\psi=O\left(s-s_2\right), \quad s\to s_2^+.
\end{equation*}
Thus $\phi$ has four singularities {\tiny $$\left\{\pm1,\sqrt{\frac{1-\sqrt{1-4c_1^2}}{2}},\sqrt{\frac{1+\sqrt{1-4c_1^2}}{2}}\right\}.$$}

If $-\frac{1}{2}<c_1<0$, we can prove that $\phi$ has four singularities {\tiny $$\left\{\pm1,-\sqrt{\frac{1-\sqrt{1-4c_1^2}}{2}},-\sqrt{\frac{1+\sqrt{1-4c_1^2}}{2}}\right\}.$$} Furthermore, the behavior of $\phi$ near the singularities is similar to the case $0<c_1<\frac{1}{2}$.

In the case $c_0=1$,  we conclude that $\phi$ can not be extended to positive smooth functions across its singularities.

\textbf{Case 2:} $c_0\neq1$. In this case,
\begin{equation*}\label{Riccati c neq 1 solution '}
  \psi'=\frac{c_1^2c_0(1-c_0)-\frac{c_1c_0s}{\sqrt{1-s^2}}-[c_1(1-c_0)s+\sqrt{1-s^2}]^2}{[c_1c_0+s(c_1(1-c_0)s+\sqrt{1-s^2})]^2},
\end{equation*}

It is not hard to show that the equation
\begin{equation*}
  c_1c_0+s(c_1(1-c_0)s+\sqrt{1-s^2})=0
\end{equation*}
has no real zero points, or a single zero of order 2, or two zeros of order 1.

Thus the properties of the singularities of $\phi$ are similar to Case 1. In this case, $\phi$ also can not be extended to positive smooth functions across its singularities.

\end{proof}
\begin{remark}
  \textnormal{One notices  that the functions (\ref{solution Riccati}) are constructed by other methods in \cite{Asanov06b,Shen09}. Using these functions, one can construct almost regular Landsberg spaces which are not Berwaldian (see \cite{Asanov06b,Shen09}).}
\end{remark}

\end{document}